\documentclass[a4paper,12pt]{article}
\usepackage[utf8x]{inputenc}
\usepackage[english]{babel}
\usepackage{graphicx}
\usepackage{amscd}
\usepackage{amsfonts}
\usepackage{latexsym}
\usepackage{amssymb,amsmath}
\usepackage[usenames]{color}
\usepackage{caption}
\usepackage{subcaption}
\usepackage{amscd}
\usepackage{psfrag}

\usepackage{amsthm}

\theoremstyle{plain}
\newtheorem{teor}{Theorem}[section]

\newtheorem{prop}[teor]{Proposition}

\newtheorem{remark}{Remark}[section]

\newtheorem*{str}{Straightening Theorem}
\theoremstyle{definition}
\newtheorem{defi}[teor]{Definition}

\def\C{\mathbb{C}}
\def\S{\mathbb{S}}
\def\D{\mathbb{D}}

\def\H{\mathbb{H}}
\def\U{\mathbb{U}}
\def\O{\mathbb{O}}

\def\l0{_{\lambda_0}}
\def\i0{_{\iota_0}}
\def\ll{_{\lambda}}
\def\wl{\widehat \lambda}
\def\wwl{_{\widehat\lambda}}

\title{Parameter space for families of parabolic-like~mappings}
\author{Luna Lomonaco}

\begin{document}

\maketitle

\begin{abstract}
In this paper we
study families of degree $2$ parabolic-like mappings $(f\ll)_{\lambda \in \Lambda}$ (as we defined in
\cite{LL}). We prove
that the hybrid conjugacies between a nice analytic family of degree $2$ parabolic-like mappings and members of the family $Per_1(1)$
induce a
continuous map $\chi: \Lambda \rightarrow \C$, which under suitable conditions
restricts to a ramified covering from the connectedness locus of $(f\ll)_{\lambda \in \Lambda}$
to the connectedness locus $M_1\setminus \{1\}$ of $Per_1(1)$. 
As an application, we prove that the connectedness locus of the family $C_a(z)=z+az^2+z^3,\,\,a\in \C$ presents baby $M_1$.
\end{abstract}
\section{Introduction}\label{intro}
 A degree $d$ polynomial-like mapping is a degree $d$ proper holomorphic map $f: U' \rightarrow U$, where $U'$ and
 $ U$ are topological disks and $U'$ is compactly contained in $U$. This definition captures the behaviour of a polynomial in a neighbourhood of its
filled Julia set. The filled Julia set is defined in the polynomial-like case as
the set of points which do not escape the domain. The external class of a polynomial-like map is the (conjugacy classes of) the map which
encodes the dynamics of the polynomial-like map outside the filled Julia set. 
 The external class of a degree $d$ polynomial-like map is a degree $d$ real-analytic orientation
preserving and strictly expanding self-covering of the unit circle: the expansivity of such a
circle map implies that all the periodic points are repelling, and in particular
not parabolic.

In order to avoid this restriction, in \cite{LL} we introduce an object, which we call
\textit{parabolic-like mapping},
 to describe the parabolic case. A parabolic-like mapping is thus similar to a polynomial-like
mapping, but with a parabolic external class; that is to say, the external map has a
parabolic fixed point.  A parabolic-like map can be seen as the union of two different dynamical
parts: a polynomial-like part and a parabolic one, which are
connected by a dividing arc.
\begin{defi}\label{definitionparlikemap} \textbf{(Parabolic-like maps)\,\,\,}
A parabolic-like map of degree $d$ is a 4-tuple ($f,U',U,\gamma$) where 
\begin{itemize}
	\item $U'$ and $U$ are open subsets of $\C$, with $U',\,\, U$ and $U \cup U'$ isomorphic to a disc, and $U'$ not
contained in $U$,
	\item $f:U' \rightarrow U$ is a proper holomorphic map of degree $d$ with a parabolic fixed point at $z=z_0$ of
 multiplier 1,
	\item $\gamma:[-1,1] \rightarrow \overline {U}$ is an arc with $\gamma(0)=z_0$, forward invariant under $f$, $C^1$
on $[-1,0]$ and on $[0,1]$, and such
that
$$f(\gamma(t))=\gamma(dt),\,\,\, \forall -\frac{1}{d} \leq t \leq \frac{1}{d},$$
$$\gamma([ \frac{1}{d}, 1)\cup (-1, -\frac{1}{d}]) \subseteq U \setminus U', \,\,\,\,\,\,\gamma(\pm 1) \in \partial U.$$
It resides in repelling petal(s) of $z_0$ and it divides $U'$ and $U$ into $\Omega', \Delta'$ and $\Omega, \Delta$
respectively, such that $\Omega' \subset \subset U$ 
(and $\Omega' \subset \Omega$), $f:\Delta' \rightarrow \Delta$ is an isomorphism and
$\Delta'$ contains at least one attracting fixed petal of $z_0$. We call the arc $\gamma$ a \textit{dividing arc}.

\end{itemize}

\end{defi}
In \cite{LL} we extend the theory of polynomial-like maps to parabolic-like maps, and we straighten degree $2$ parabolic-like maps to members
of the family of quadratic rational maps with a parabolic fixed point of multiplier $1$ at infinity and critical points at $1$ and $-1$,
which is $$Per_1(1)= \{[P_A]\,|\,P_A(z)=z+\frac{1}{z}+A\}.$$ More precisely, we prove the following:
\begin{str}
 Every degree $2$ parabolic-like mapping $(f,U',U,\gamma)$ is hybrid equivalent
to a member of the family $Per_1(1)$. Moreover, if $K_f$ is connected, this member is
unique.
\end{str}
Note that $[P_A]= \{ P_A,\,\,P_{-A}\},$ since the involution $z \rightarrow -z$ conjugates $P_{A}$ and $P_{-A}$, interchanging the roles of the critical points. 
We refer to a member of the family $Per_1(1)$ as one of the representatives of
its class.
The family $Per_1(1)$ is typically parametrized by $B=1-A^2$, which is the multiplier of the 'free' fixed point
$z=- 1/A$ of $P_A$. The connectedness
locus of $Per_1(1)$ is called $M_1$.
If $\textbf{f}=(f_{\lambda}: U_{\lambda}' \rightarrow U_{\lambda})_{\lambda \in \Lambda} $ is a
family of degree $2$ parabolic-like maps with parameter space $\Lambda \subset \C$,
calling $M_f$ the connectedness locus of $\textbf{f}$, by the uniqueness of the Straightening
we can define a map $$\chi:  M_f \rightarrow M_1$$
$$\hspace{1cm} \lambda \rightarrow B,$$
which associates to each $\lambda$ the multiplier of the fixed point $z=- 1/A$ of the
member $[P_A]$ hybrid equivalent to $f_{\lambda}$.

In this paper we will prove that if the family $\textbf{f}$ is \textit{analytic} and \textit{nice} (see Def. \ref{def} and \ref{nf}),
\textit{the map $\chi$ extends to a map defined on the whole of $\Lambda$
(see \ref{external}), whose
restriction
to
$M_f$, under suitable conditions (see Def. \ref{properfam})
is a ramified covering of $M_1\setminus \{1\}$ (see Thm. \ref{bigthm})}.
The reason why the map $\chi$ covers
$M_1\setminus \{1\}$, instead of the whole of $M_1$, resides in the definition of analytic family of parabolic-like
mappings, and it will be explained in section \ref{anfam}.
\begin{figure}
  \centering
  \begin{minipage}{.4\textwidth}
  \centering
  \includegraphics[width= 5cm]{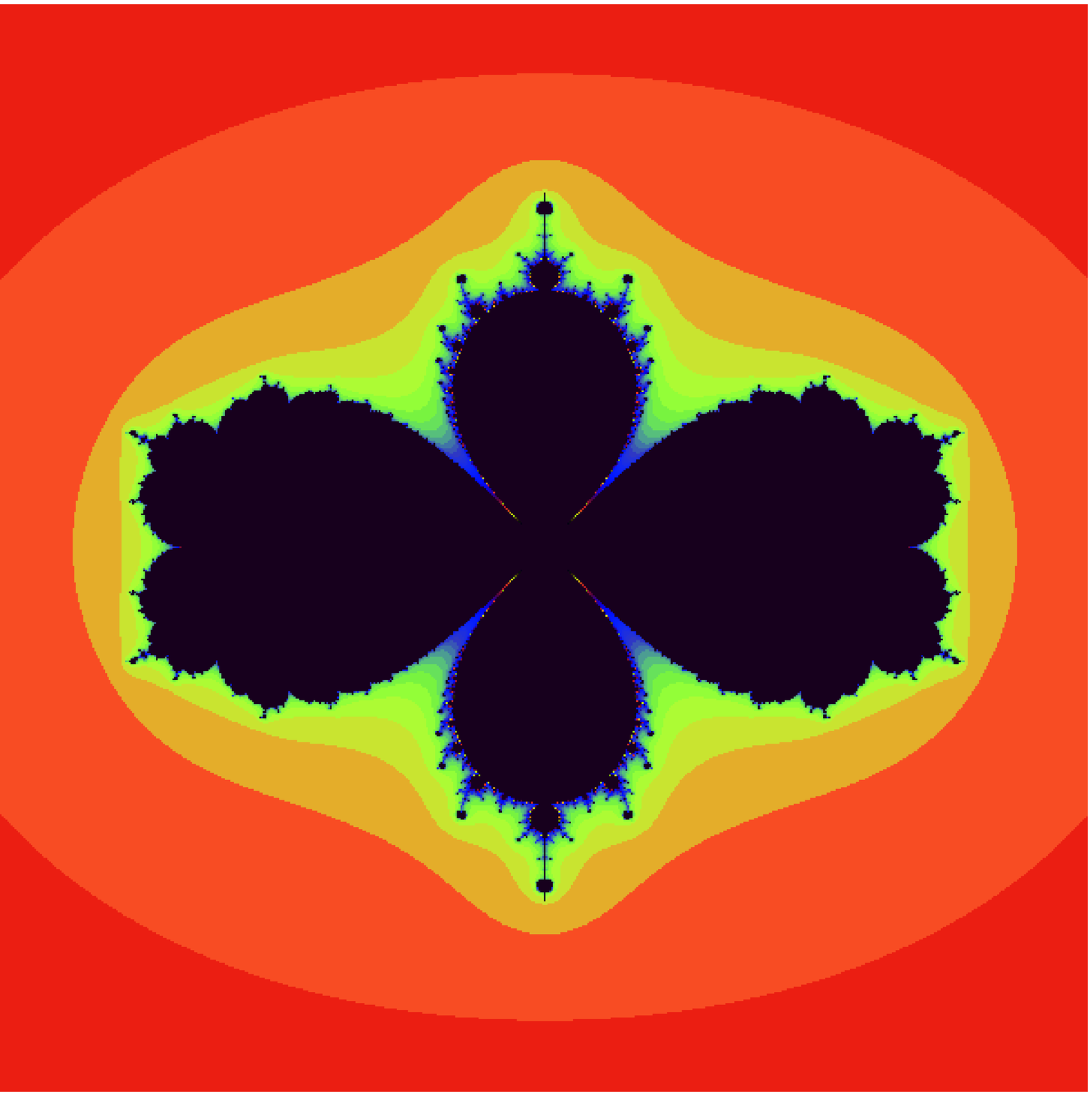}
  \captionof{figure}{\small Connectedness locus of the family $C_a$.}
  \label{M}
  \end{minipage} 
  \hspace{1.5cm}
  \begin{minipage}{.4\textwidth}
  \centering
  \includegraphics[width= 5cm]{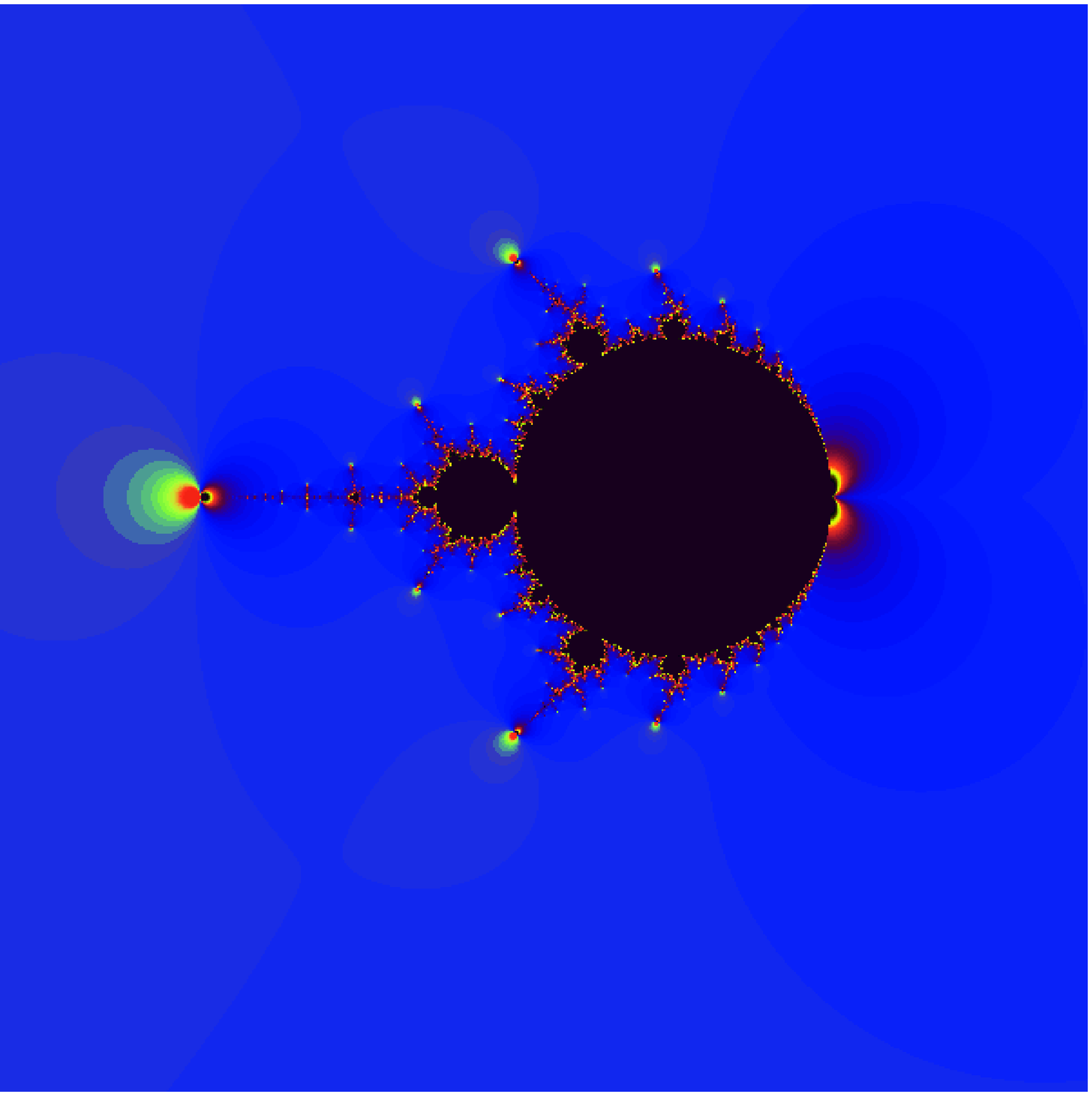}
  \captionof{figure}{\small Connectedness locus $M_1$ of the family $Per_1(1)$.}
  \label{M1}
  \end{minipage}

\end{figure}
As an application, we will show that the connectedness locus of the family $C_a(z)= z+az^2+z^3,\,\,a\in \C$ (see Fig.\ref{M}) presents
$2$ baby $M_1$ (see Fig.~\ref{M1}).\\

The results in this paper were developed during the author's Ph.d. So the
author would like to thank her former advisor, Carsten Lunde Petersen, for suggesting the idea of parabolic-like mapping and for his help,
Roskilde University and Universit\'e Paul Sabatier for their hospitality, and Roskilde
University, the ANR-08-JCJC-0002 founded by the Agence Nationale de la Recherche and the Marie Curie RTN 035651-CODY for
their financial support during her Ph.d. 

\section{Definitions and statement of the result}
In this Section we define an analytic family of parabolic-like maps and its connectedness locus, nice families of parabolic-like maps,
and we give an example
of nice analytic family of parabolic-like maps. Then we give a review of the Straightening Theorem,
an overview of this paper and we state the main result.

\begin{defi}\label{def}
 
Let $\Lambda \subset \C$, $\Lambda\approx \D$ and let
$\textbf{f}=(f_{\lambda}: U_{\lambda}' \rightarrow U_{\lambda} )_{\lambda \in \Lambda}$ be a
family of degree $d$ parabolic-like mappings. Set $\U'= \{(\lambda, z)|\,\, z \in
U_{\lambda}' \}$, $\U= \{(\lambda, z)| \,\, z \in U_{\lambda} \}$, $\O'= \{(\lambda, z)| \,\, z \in
\Omega_{\lambda}' \}$, and $\O= \{(\lambda, z)| \,\,  z \in \Omega_{\lambda} \}$.
Then  $\textbf{f}$ is a degree $d$ analytic family of parabolic-like maps if the following conditions are
satisfied: 
\begin{enumerate}
	\item $\U'$, $\U$, $\O'$ and $\O$ are domains in $\C^2$;
	\item the map $\textbf{f}: \U' \rightarrow \U$ is holomorphic in $(\lambda,z)$. 
        \item all the parabolic-like maps in the family have 
        the same number of attracting petals in the filled Julia set.
\end{enumerate}
 \end{defi}
For all $\lambda \in \Lambda$ let us call $z\ll$ the parabolic-fixed point of $f\ll$, and let us set 
$K_{\lambda}=K_{f_{\lambda}}$, and $J_{\lambda}=J_{f_{\lambda}}$. 
Define $$M_f=\{\lambda \,\, |\,\,K_{\lambda} \,\, \mbox{is connected} \}.$$

\subsubsection{Nice families}\label{nf}
An analytic family of parabolic-like mappings is
\textit{nice} if there exists a 
holomorphic motion of the
dividing arcs $$\Phi: \Lambda \times \gamma\l0 \rightarrow \C,$$ and there exists a holomorphic motion of the ranges
$$B: \Lambda \times \partial U\l0 \rightarrow \C$$
which is a piecewise $C^1$-diffeomorphism with no cusps in $z$ (for every fixed $\lambda$), and
$B\ll(\gamma\l0(\pm 1))=\gamma\ll(\pm 1)$.
\subsubsection{Remarks about the definition and motivations}
A nice family is basically endowed by definition with a holomorphic motion of a fundamental annulus (see Section \ref{holmot}).
We did not require analytic families to have these properties, because the concept of parabolic-like map is local.
On the other hand, since all the maps in an analytic family of parabolic-like maps have the same number of attracting petals in
its filled Julia set, it follows from the holomorphic parameter dependence of Fatou coordinates (see Appendix in \cite{Sh}),
that in many cases there is a  holomorphic motion of the dividing arcs
(however, in individual cases further detail might be required according to circumstances).
Moreover, since the concept of parabolic-like map is local, in many cases it is not difficult to construct a holomorphic motion of the ranges for 
an analytic family of parabolic-like mappings.

\subsubsection{Degree $2$ analytic families of parabolic-like maps}\label{anfam}

The definition of analytic family of parabolic-like maps is valid for any degree. However, since in this paper we are interested in proving
that, under suitable conditions, the map $\chi$ defined in the introduction is a ramified covering between $M_f$ and $M_1 \setminus \{1\}$, 
in the remainder we will restrict our attention to degree $2$ nice analytic
families of
parabolic-like maps. 
All the maps of an analytic family of
parabolic-like maps have the same number of
attracting petals in their filled Julia set, and each
(maximal) attracting petal requires a critical point in its boundary. Hence, if $\textbf{f}$ is a degree $2$ analytic family of
parabolic-like maps, either for each $\lambda \in \Lambda$ the map $f\ll$ has
no attracting petals in $K\ll$, or for each $\lambda \in \Lambda$ the map $f\ll$ has
exactly one attracting petal in $K\ll$. 

Consider now the family $Per_1(1)$. The $\Delta$-part of a parabolic-like mapping requires (at least) one attracting petal,
and for all the members of the family $Per_1(1)$ with $A \neq 0$ the parabolic fixed point has parabolic multiplicity $1$. So
a parabolic-like restriction of $P_A$, with $A\neq 0$ has no attracing petals in the filled Julia set. On the other hand,
$P_0= z + 1/z$ has a parabolic fixed point of parabolic multiplicity $2$ and the Julia set of $P_0$ is the
common boundary of the immediate parabolic basins, so a parabolic-like restriction of $P_0$ has exactly one attracting petal in its
filled Julia set.

So, if all the members an analytic family of degree $2$ parabolic-like mappings $\textbf{f}$ have exactly one attracting petal in the filled Julia set,
they are all hybrid conjugate to the map $P_0= z + 1/z$, and $\chi(\lambda)\equiv 1,$
(but this case is not really interesting).
On the other hand, if all the members of $\textbf{f}$ have no petals in the filled
Julia set, there is no $\lambda \in \Lambda$ such that $f\ll$ is hybrid conjugate to the map
$P_0= z + 1/z$, and finally the image of $M_f$ under the map $\chi$
is not the whole of $M_1$, but it belongs to $M_1\setminus \{1\}$. This is the case we are interested in.

\subsection{Example}
Consider the family of cubic polynomials
$$C_a(z)=z+az^2+z^3,\,\,\,a \in \C.$$
The maps belonging to this family have a parabolic fixed point at $z=0$ of multiplier $1$, and
critical points at
$c_+(a)=\frac{-a + \sqrt{ a^2 -3}}{3} $ and $c_-(a)=\frac{-a - \sqrt{ a^2 -3}}{3} $.
Call $\mathcal{C}$ the connectedness locus for this family. Let $\varphi_a$
denote the B\"ottcher coordinates for $C_a$ tangent to the identity at infinity, call $\tilde c_-(a)$ the co-critical point of $c_-(a)$ and let
$$\Phi: \C \setminus \mathcal{C} \rightarrow \C \setminus \overline{\D}$$
be the conformal rappresentation of $\C \setminus \mathcal{C}$ given by $$\Phi(a)=\varphi (\tilde c_-(a)).$$
Define
$\Lambda \subset \mathcal{C}$ as
the open set bounded by the external rays of angle $1/6$ and $2/6$ (see \cite{N}). In this section we are going to
prove that the family $(C_a(z)=z+az^2+z^3)_{a\in \Lambda}$ yields to a nice family of
parabolic-like mappings.

\subsection{For every $a \in \Lambda$, $C_a$ presents a parabolic-like restriction}\label{dynamicalconstruction}
Let us construct a parabolic-like restriction for every member of the family $(C_a)_{a \in \Lambda}$. 
Call
$\Xi_a$
the immediate basin of attraction of the parabolic fixed point $z=0$. Then
$c_+(a)$ belongs to $\Xi_a$, while $c_-(a)$ does not belong to $\Xi_a$.
Let $\phi_a : \Xi_a \rightarrow \D$ be
the Riemann map normalized by setting $\phi_a(c_+(a))=0$ and $\phi_a(z)\stackrel{z\rightarrow
0}\longrightarrow 1$, and let
$\psi_a:\D \rightarrow \Xi_a$ be its
inverse. By the Carathéodory theorem the map $\psi_a$ extends
continuously to $\S^1$. Note that $\phi_a  \circ C_a \circ
\psi_a=h_2$. Let $w_a$ be a $h_2$ periodic point in the first
quadrant, such that the hyperbolic geodesic $\widetilde{\gamma_a} \in \D$ connecting $w_a$ and $\overline{w_a}$
separates the
critical value $z=1/3$ from the parabolic fixed point $z=1$. Let $U_a$ be the Jordan domain bounded by
$\widehat{\gamma_a}=
\psi_a(\widetilde{\gamma_a})$, union the arcs up to potential level $1$ of the external rays landing at
$v_a=\psi_a(w_a)$
and
$\bar v_a=\psi_a(\overline{w_a})$, together with the arc of the level $1$ equipotential connecting this two rays around
$c_-(a)$ (see
Fig. \ref{cubicbcnnew}). Let
$U_a'$ be the connected component of $C_a^{-1}(U_a)$ containing $0$ and the dividing arcs $\gamma_{a \pm}$ be the fixed external rays landing at
the
parabolic
fixed point
$0$ and parametrized by potential. Then
($C_a,U_a',U_a,\gamma_{a} $) is a parabolic-like map of degree $2$ (see Fig.
\ref{cubicbcnnew}).

\begin{figure}[hbt!]
\centering
\psfrag{gp}{$\gamma_+$}
\psfrag{g-}{$\gamma_-$}
\psfrag{U}{$U$}
\psfrag{U'}{$U'$}
\psfrag{O}{$\Omega$}
\psfrag{O'}{$\Omega'$}
\psfrag{oa}{$\Xi_a$}
\psfrag{-a}{$c_+(a)$}
\psfrag{-a/2}{$0$}
\psfrag{a}{$c_-(a)$}
\psfrag{f-1w}{$\phi^{-1}(w)$}
\psfrag{f-1wb}{$\phi^{-1}(\overline{w})$}
\psfrag{fi}{$\phi$}
\psfrag{f-a}{$\phi(c_+(a))$}
\psfrag{1/3}{$1/3$}
\psfrag{w}{$w$}
\psfrag{wb}{$\overline{w}$}
\psfrag{fa}{$\phi(0)$}
\includegraphics[width= 14cm]{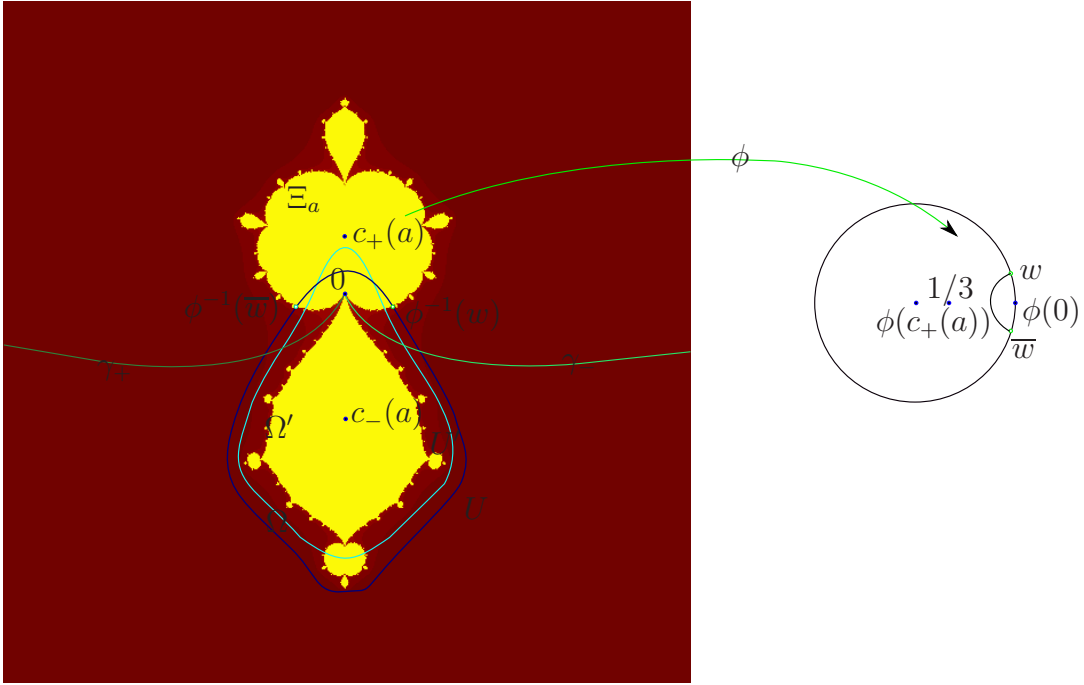}
\caption{\small Construction of a parabolic-like restriction of a member $C_a$ of the family $(C_a(z)=z+az^2+z^3)_{a \in
\Lambda}$. In the picture, $a=i$.}
\label{cubicbcnnew}
\end{figure}

\subsection{The family $(C_a(z)=z+az^2+z^3)_{a \in \Lambda}$ yields to a nice analytic family of
parabolic-like mappings}
For every $a \in \Lambda$ the parabolic fixed point $0$ of $C_a$ has parabolic multiplicity $1$, and
($C_a,U_a',U_a,\gamma_{a} $) is a parabolic-like map with no attracing petals in its filled Julia set.
By the construction we gave, it follows easily that $(C_a)_{a \in \Lambda}$ restricts to
an analytic                                 
family of parabolic-like mappings. 
Since external rays move holomorphically, to prove that this analytic family of parabolic-like maps is nice
it suffices to show that the boundaries of $U_a$ move holomorphically with the
parameter (by construction the motion defines a piecewise $C^1$-diffeomorphisms with no cusps in $z$).
Let us start by proving that the basin of attraction $\Xi_a$ of $0$ depends holomorphically on the parameter.
\subsubsection{The basin of attraction of the parabolic fixed point depends
holomorphically on $a$}
Call $\mathcal{P}_a$ the maximal attracting petal in $\Xi_a$, and let $F_a: \mathcal{P}_a \rightarrow \H_l$ be Fatou coordenates for $C_a$
normalized by sending the critical point $c_+(a)$ to $1$.
Since the family $(C_a)_{a \in \C}$ depends holomorphically on $a$, $F_a$ depends holomorphically on $a$
and the extended Fatou coordenates to the whole parabolic basin
$\mathcal{F}_a:\Xi_a \rightarrow \C$ depend holomorphically on $a$. 
On the other hand, let $\Phi_h: \D \rightarrow \C$ be extended Fatou coordinates
for
the map $h_2$, normalized by sending the
critical point to $1$.
Since the Riemann map $\phi_a$ is a holomorphic conjugacy between $C_a$ and $h_2$,
$\Phi_h \circ \phi_a $ are Fatou coordinates for $C_a$.
Since $\Phi_a(c_+(a))= 1=\Phi_h \circ \phi_a(c_+(a))= \Phi_h(0)=1$, we have that $\Phi_a=\Phi_h \circ \phi_a $. Hence
the Riemann map
$\phi_a$ depends holomorphically on $a$. So (fixing a base point $a_0 \in \Lambda$) the dynamical holomorphic motion
$\Phi_a^{-1}\circ \Phi_{a_0}: \Lambda \times \Xi_{a_0} \rightarrow \C$ (holomorphic in $z$) induces a dynamical holomorphic
motion $\phi_a^{-1}\circ \phi_{a_0}: \Lambda \times \Xi_{a_0} \rightarrow \C $ (holomorphic in $z$), which extends 
by the $\lambda$-Lemma to a dynamical holomorphic motion of $\overline{\Xi}_{a_0}$.
\subsubsection{The family $(C_a(z)=z+az^2+z^3)_{a \in \Lambda}$ restricts to a nice analytic family of
parabolic-like mappings}
Since
$\overline \Xi_a$ moves holomorphically, the points $v_a$ and $\bar v_a$
and the arc $\widehat{\gamma_a}$ defined in \ref{dynamicalconstruction}
depend holomorphically on $a$. Since equipotentials and external
rays move holomorphically, for every $a \in \Lambda$ the set $\partial U_a$ moves holomorphically. Hence the family 
$(C_a(z)=z+az^2+z^3)_{a \in \Lambda}$ restricts to a degree $2$ nice analytic family of parabolic-like maps.

\subsection{Review and overview}\label{anfam}
In \cite{LL} we proved that a degree $2$ parabolic-like map is hybrid conjugate to a member of the family
$Per_1(1)$ by changing its external class into the class of $h_2(z)=\frac{z^2+1/3}{z^2/3+1}$ (see Theorem 6.3 in \cite{LL}) and showing 
that a parabolic-like map is holomorphically conjugate to a member of the family $Per_1(1)$ if and only
if its external class is given by the class of $h_2$ (see Proposition 6.2 in \cite{LL}).
We defined a (quasiconformal) conjugacy between two parabolic-like maps $(f,U'_f,U_f,\gamma_f)$ and $(g,U'_g,U_g,\gamma_g)$
to be a (quasiconformal) homeomorphism
between (appropriate) restrictions of $U_f$ and $U_g$ which conjugates dynamics on $\Omega_f\cup\gamma_f$ (see Def. $3.3$ in \cite{LL}).
Let us review how we changed the external class of a degree $2$ parabolic-like map $f$ into the class of $h_2$.
As first step, we constructed a homeomorphism
$\widetilde{\psi}$, quasiconformal everywhere but at the parabolic fixed point,
between a fundamental annulus $A_f= \overline U_f \setminus \widetilde \Omega_f'$ of $f$ and
a fundamental annulus $A=\overline B \setminus \widetilde \Omega_B'$ of $h_2$. 
Then we defined on $A_f$ an almost complex structure $\sigma_1$
by pulling back the standard structure by $\widetilde{\psi}$.
In order to  obtain on $U_f$ a bounded and invariant (under a map coinciding with $f$ on $\Omega_f$)
almost complex structure $\sigma$ we replaced $f$
with $h_2$ on $\Delta$, and spread $\sigma_1$ by the dynamics of this new map $\tilde f$ (and kept the standard
structure on $K_f$).
Finally, by integrating $\sigma$ we obtained a parabolic-like map
hybrid conjugate to $f$ and with external map $h_2$.

In this paper we want to perform this surgery for nice analytic families of degree $2$ parabolic-like maps, and prove
that the map $\chi:M_f \rightarrow M_1$ induced by the family of hybrid conjugacies extends to a continuous map $\chi:\Lambda \rightarrow \C$
which under suitable conditions restricts to a branched covering of $M_1 \setminus\{1\}$.
We will start by defining a family of quasiconformal maps, depending
holomorphically on the parameter, between a
fundametal annulus of $h_2$ and fundamental annuli $A\ll= \overline U\ll \setminus \widetilde\Omega'\ll$ of $f\ll,\,\,\lambda \in \Lambda$. In
analogy with the polynomial-like setting we will
call this family a \textit{holomorphic Tubing}.
In order to construct a holomorphic Tubing, fixed a $\lambda_0 \in \Lambda$, we will start by constructing
a quasiconformal homeomorphism $\widetilde{\psi}$
between $A$ and $A\l0$ (see Section \ref{psidiff}) and a
dynamical holomorphic motion $\widehat{\tau}:\Lambda \times A\l0 \rightarrow A\ll$ (see Section \ref{holmot}).
Hence we will obtain a holomorphic Tubing by composing the inverse of $\widetilde{\psi}$
with the holomorphic motion (see Section \ref{A}).
By Tubing, we will extend the map $\chi$ to the whole of $\Lambda$ (see Section \ref{external}).
We will prove that the map $\chi$ is continuous (see Section \ref{con}), holomorphic on the interior
of $M_f$ (see Section \ref{anal}) and with discrete fibers (see Section \ref{d}). Finally, we will
prove that, on compact subsets of $\Lambda$, the map $\chi$ is a degree $\mathcal{D}>0$ branched covering (see Section \ref{final}).
By defining proper families of parabolic-like maps we wil give the condition under which, for
each neighborhood $U$ of $1$, $\chi^{-1}(M_1 \setminus U)$ is a compact subset of $\Lambda$ (see Section \ref{p}).
This implies the following result:
\begin{teor}\label{bigthm}
 Given a proper family of parabolic-like maps $(f_{\lambda})_{\lambda \in \Lambda \approx \D}$, the map
 $\chi: M_f \rightarrow M_1\setminus\{1\}$ is a degree $\mathcal{D}>0$ branched covering. More precisely,
 for every neighborhood $U$ of $1$ in $\C$ (with $ 0 \notin U$)
there exists a neighborhood $\hat V$
of $M_1 \setminus U$ in $\chi(M_f)$ such that the map
$\chi: \chi^{-1}( \hat V)
\rightarrow  \hat V$ is a degree $\mathcal{D}>0$ branched covering.
\end{teor}

\section{Holomorphic motion and Tubings}
In this section we will
first, fixed a
$\lambda_0 \in \Lambda$, construct a fundamental annulus for $h_2$ and one for $f\l0$, and recall the
quasiconformal $C^1$-diffeomorphism $\widetilde{\psi}$ between these fundamental annuli. Then we will
construct a holomorphic motion of the ranges of the nice analytic family of parabolic-like maps, and by this holomorphic motion
derive fundamental annuli for
$f\ll$ from the fundamental annulus of $f\l0$.
Finally we will obtain a holomorphic Tubing by composing the inverse of $\widetilde{\psi}$
with the holomorphic motion.

\subsubsection{A fundamental annulus $A$ for $h_2$}
The map $h_2(z)=\frac{z^2 +1/3}{1+z^2/3}$ is an external map of every member of the family $Per_1(1)$ (see Prop. 4.2 in \cite{LL}). 
Let $h_2: W' \rightarrow W$ (where $W= \{z : \exp(-\epsilon)<|z|< \exp(\epsilon)\}$ for an $\epsilon>0$, and
$W'=h_2^{-1}(W)$) be a degree $2$ covering extension (this is, an extension such that $h_2: W' \rightarrow W$ is a degree $2$ covering
and there exists a dividing arc which devides $W'\setminus \D$ and $W'\setminus \D$ into $\Omega'_W,\,\Delta'_W$ and $\Omega'_W,\,\Delta'_W$ respectively,
such that $\Omega_W \setminus \Omega'_W$ is a topological quadrilateral; see Def. 5.2 in \cite{LL}).
Choose $\lambda_0 \in \Lambda$.
Let $h\l0$ be an external map of $f\l0$, $z_0$ be its parabolic fixed point and define 
$\gamma_{h\l0 +}=\alpha\l0(\gamma_{\lambda_0 +})$, $\gamma_{h\l0 -}=\alpha\l0(\gamma_{\lambda_0 -})$ (where $\alpha$ is
an external equivalence between $f\l0$ and $h\l0$).

Let $\Xi_{h_f +}$ and $\Xi_{h_f -}$ be repelling petals for the parabolic fixed point $z_0$ which intersect the unit circle and
$\phi_{\pm}: \Xi_{h_f \pm} \rightarrow \H_l  $ be Fatou coordinates for $h\l0$ with axis tangent to the unit
circle at the parabolic fixed point $z_0$. 
Let $\Xi_{h +}$ and $\Xi_{h -}$ be repelling petals which intersect the unit circle for the parabolic fixed point $z=1$ of $h_2$, and
let $\widetilde \phi_{\pm}: \Xi_{h \pm} \rightarrow \H_l  $ be Fatou coordinates for $h_2$ with axis tangent to the unit
circle at $1$. Define  $\widetilde{\gamma}_{+}= \widetilde{\phi}_{+}^{-1}(\phi_{h\l0 +}(\gamma_{h\l0 +}))$ and 
$\widetilde{\gamma}_{-}= \widetilde{\phi}_{-}^{-1}(\phi_{h\l0 -}(\gamma_{h\l0 -}))$. 

Define $\widetilde \Delta_W=h_2(\Delta_W \cap \Delta_W'),\,\,\,\, \widetilde W=\Omega_W \cup \widetilde{\gamma} \cup
\widetilde \Delta_W,\,\,\,\widetilde W'=h_2^{-1}(\widetilde W), \,\,\,\widetilde \Omega_W'=
\Omega_W' \cap \widetilde W',\,\,\,\widetilde \Delta_W'=\Delta_W' \cap \widetilde W'$ and $Q_W=\Omega_W \setminus
\widetilde \Omega_W'$. 
We call \textit{fundamental annulus for $h_2$} the topological annulus $A= \overline W \setminus (\widetilde \Omega'_W
\cup \D)$.
\subsubsection{A fundamental annulus $A\l0$ for $f\l0$ and the 
map $\widetilde{\psi}: A\l0 \rightarrow A$}\label{psidiff}

Let $\Phi_{\Delta\l0}: \overline \Delta\l0 \rightarrow \overline \Delta_W$
 be a homeomorphism coinciding with
$\widetilde{\phi}_{\pm}^{-1}\circ \phi_{h\l0 \pm} \circ \alpha\l0$ on $\gamma_{\lambda_0 \pm}$,
quasiconformal on $\overline \Delta\l0 \setminus \{z\l0\}$
(where $z\l0$ is the parabolic fixed point of $f\l0$) and real-analytic diffeomorphism
on $\Delta\l0$
(see Claim $6.1$ in the proof of Thm. $6.3$ in \cite{LL}). Define
$\widetilde{\Delta}\l0= \Phi_{\Delta\l0}^{-1}(\widetilde \Delta_W)$,
$\widetilde{\Delta'}\l0=\Phi_{\Delta\l0}^{-1}(\widetilde \Delta_W')$, and $\widetilde{U}\l0= (\Omega\l0  \cup \gamma\l0
\cup
\widetilde{\Delta}\l0) \subset U\l0$.
Consider
$$ \widetilde{f}\l0(z)=\left\{
\begin{array}{cl}
\Phi_{\Delta\l0}^{-1} \circ h_2 \circ \Phi_{\Delta\l0} &\mbox{on  } \widetilde{\Delta'}\l0 \\
f\l0  &\mbox{on  }  \Omega\l0'\cup \gamma\l0\\
\end{array}\right.
$$
Define $\widetilde{U'}\l0=\widetilde{f}\l0^{-1}(\widetilde{U}\l0)$, $Q\l0= \Omega\l0 \setminus \overline
{\widetilde{\Omega'}}\l0$,
and the \textit{fundamental annulus} $A\l0=\overline  U\l0 \setminus \widetilde{\Omega'}\l0$.

Let $\Phi_{Q\l0}: \overline Q\l0 \rightarrow \overline Q_W$ be a quasiconformal map which coincides
with $\widetilde{\phi}_{\pm}^{-1}\circ \phi_{h\l0 \pm} \circ \alpha\l0$ on
$\gamma_{\lambda_0 \pm}$ (see the proof Claim $6.2$ in Thm. $6.3$ in \cite{LL}).
 Define a
map $\widetilde{\psi}: A\l0\rightarrow A$ as follows :
$$ \widetilde{\psi}(z)=\left\{
\begin{array}{cl}
\Phi_{\Delta\l0}  &\mbox{on  } \overline\Delta\l0\\
\Phi_{Q\l0}  &\mbox{on  } \overline Q\l0 \\
\end{array}\right.
$$ 
The map $\widetilde{\psi}$ is a homeomorphism, quasiconformal on $A\l0 \setminus \{z\l0\}$,
so the map
$\widetilde{\Psi}:=\widetilde{\psi}^{-1}:A \rightarrow A\l0$ is a homeomorphism, quasiconformal on $ A \setminus \{1\}$.

\subsubsection{Holomorphic motion of the fundamental annulus $A\l0$}\label{holmot}
Define for all $\lambda \in \Lambda$ the set $a\ll=\overline U\ll \setminus \Omega'\ll$. Then the set
$a\ll$ is a topological annulus.
Define the map $\widetilde{\tau}: \Lambda \times \partial a\l0 \rightarrow \partial a\ll$ as follows:
$$ \widetilde{\tau}(z)=\left\{
\begin{array}{cl}
\Phi\ll  &\mbox{on  } \gamma\l0\\
B\ll &\mbox{on  } \partial U\l0 \\
f\ll^{-1}\circ B\ll \circ f\l0 &\mbox{on  } \partial U'\l0 \cap \partial \Omega'\l0\\
\end{array}\right.
$$ 
Since $\Phi\ll$ and $B\ll$ are holomorphic motions with disjoint
images on $\partial a\l0 \setminus \{\gamma_{\lambda_0}(1),\gamma_{\lambda_0}(-1)\}$,
and $f\ll: \partial U'\ll \rightarrow \partial U\ll$ is a degree $d$ covering, $\widetilde{\tau}$ is a
holomorphic motion with basepoint $\lambda_0$.
Since $\Lambda\approx \D$, by the Slodkowski's Theorem we can extend $\widetilde{\tau}$ to a holomorphic motion
$\widetilde{\tau}: \Lambda\times \widehat{\C} \rightarrow \widehat{\C}$. In particular we obtain a holomorphic motion of
$\widetilde{U}\l0$. For every $\lambda \in \Lambda$, define
$\widetilde{U}\ll=\widetilde{\tau}(\widetilde{U}\l0)$,
and $\widetilde{\Delta'}\ll=\widetilde{\tau}(\widetilde{\Delta'}\l0)$. Define for every
$\lambda \in \Lambda$ the map $\widetilde{f}\ll$ as follows:
$$ \widetilde{f}\ll(z)=\left\{
\begin{array}{cl}
\widetilde{\tau} \circ \widetilde{\Psi}\circ  h_2 \circ \widetilde{\Psi}^{-1} \circ \widetilde{\tau}^{-1} &\mbox{on  }
\widetilde{\Delta'}\ll \\
f\ll  &\mbox{on  }  \Omega\ll'\cup \gamma_{f\ll}\\
\end{array}\right.
$$
and the set
$\widetilde{U}'\ll=\widetilde{f}\ll^{-1}(\widetilde{U}\ll)$.
Finally, define for all $\lambda \in \Lambda$ the set $A\ll= \overline{U}\ll \setminus  \widetilde\Omega'\ll$. Then the set
$A\ll$ is a topological annulus, and we call it \textit{fundamental annulus of $f\ll$}.
The holomorphic motion $\widetilde{\tau}: \Lambda\times \widehat{\C} \rightarrow \widehat{\C}$ restricts to a holomorphic
motion $$\widehat{\tau}:\Lambda \times A\l0 \rightarrow A\ll$$
which respects the dynamics. 
\subsubsection{Holomorphic Tubings}\label{A}
Define $T:=\widehat{\tau}\circ \widetilde{\Psi}: \Lambda\times A \rightarrow A\ll$. We call the map $T$ a \textbf{holomorphic Tubing}. A holomorphic Tubing is not a
holomorphic motion, since $T\l0=\widetilde{\Psi}\neq Id$, but nevertheless it is
quasiconformal in $z$ for every fixed $\lambda \in \Lambda$ and holomorphic in $\lambda$ for every fixed $z \in A$.
\subsubsection{Straightening of the members of the family $(f\ll)_{\lambda\in \Lambda}$}
Let us now straighten the members of the family $(f\ll)_{\lambda\in \Lambda}$ to members of the family $Per_1(1)$.
For every $\lambda \in \Lambda$ define on $U\ll$ the Beltrami form $\mu\ll$ as follows:
$$ \mu_{\lambda}(z)=\left\{
\begin{array}{cl}
\mu_{\lambda,0}=(T \ll)_*(\sigma_0) &\mbox{on  } A\ll \\
\mu_{\lambda,n}=(\widetilde f_{\lambda}^{n})^*\mu_{\lambda,0} &\mbox{on  } (\widetilde f_{\lambda})^{-n}(A\ll)\\
0 &\mbox{on  } K_{\lambda}\\
\end{array}\right.
$$
For every $\lambda$ the map $T\ll$ is quasiconformal, hence
$||\mu_{\lambda,0}||_{\infty} \leq k<1$ on every compact subset of $\Lambda$. On $\widetilde \Omega'_{\lambda}$
the Beltrami form $\mu_{\lambda,n}$
is obtained by spreading $\mu_{\lambda,0}$ by the
dynamics of $f\ll$, which is holomorphic, while on $\Delta\ll$ the Beltrami form $\mu_{\lambda,n}$ is constant for all
$n$ (by construction of the map $\widetilde f_{\lambda}$).
 Thus $||\mu_{\lambda}||_{\infty}=||\mu_{\lambda,0}||_{\infty}$, which is
bounded. By the measurable Riemann mapping theorem (see \cite{Ah}) for every $\lambda \in \Lambda$ there exists
a quasiconformal map $\phi\ll: U\ll \rightarrow \D$ such that $(\phi\ll)^*\mu_0=\mu\ll$. Finally, for
every $\lambda \in \Lambda$ the map
$g\ll=\phi\ll \circ f\ll \circ \phi\ll^{-1}$ is a parabolic-like map hybrid conjugate to $f\ll$ and
holomorphically conjugate to a member of the family $Per_1(1)$ (see Prop. 6.2 in \cite{LL}).
\begin{remark}\label{locboundil}
 Note that for every $\lambda \in \Lambda$, the dilatation of the integrating map $\phi\ll$ is equal to the dilatation
of the holomorphic Tubing $T\ll$. So the family of integrating maps $(\phi\ll)_{\lambda\in\Lambda}$ has locally bounded dilatation.
\end{remark}

\subsubsection{Lifting Tubings}\label{Li}
Let us lift the Tubing $T\ll$. Define $A_{\lambda,0}= A\ll$,
$B_{\lambda,1}= 
\widetilde f\ll^{-1}(A_{\lambda,0})$, $A_0= A$ and $B_1=
h_2^{-1}(A_0)$. Hence $\widetilde f\ll :B_{\lambda,1} \rightarrow  A_{\lambda,0}$ and $h_2: B_1 \rightarrow A_0$ are
degree $2$ covering maps, and we can lift the Tubing $T\ll$ to
$T_{\lambda,
1}:= \widetilde f\ll^{-1} \circ T\ll \circ h_2: B_1 \rightarrow B_{\lambda,1}$ (such that $T_{\lambda,1}=T_{\lambda}$
on $B_1 \cap A_0$). 

Define recursively $A_{\lambda,n}= B_{\lambda,n} \cap \widetilde U$, $B_{\lambda,n+1}= 
\widetilde f\ll^{-1}(A_{\lambda,n})$, $A_n= B_n \cap \widetilde W$ and $B_{n+1}=
h_2^{-1}(A_n)$. Hence $\widetilde f\ll :B_{\lambda,n+1} \rightarrow  A_{\lambda,n}$ and $h_2: B_{n+1} \rightarrow A_n$
are
degree $2$ covering maps, and we can lift the Tubing to $T_{\lambda,
n+1}:= \widetilde f\ll^{-1} \circ T_{\lambda,n} \circ h_2: B_{n+1} \rightarrow B_{\lambda,n+1}$ (such that
$T_{\lambda,n+1}=T_{\lambda,n}$ on $B_{n+1} \cap B_n$).

 In the case
$K\ll$ is connected, we can lift the Tubing $T\ll$ to the whole of $(W\cup W')\setminus \overline
\D$. If $K\ll$ is not connected, the
maximum domain we can lift the Tubing $T\ll$ to is $B_{n_0}$, such that $B_{\lambda,n_0}$ contains the critical value of
$f\ll$. Note that the extension is still quasiconformal in $z$.

\section{Properties of the map $\chi$}   
    Consider the map $\chi:  M_f \rightarrow M_1 \setminus \{1\}$ (defined in Section \ref{intro})
which associates to each $\lambda \in  M_f$ the multiplier of the fixed point of the map $P_A$ hybrid
equivalent to $f_{\lambda}$. In this section, we will first extend the map $\chi$ to the whole of $\Lambda$ (see Section \ref{external}), then
prove that the map $\chi:\Lambda \rightarrow \C $ is continuous at the boundary of $M_f$ (see Prop. \ref{con}) and that it depends analytically
on $\lambda$ for $\lambda
 \in \mathring{M_f}$ (see Prop. \ref{an}). Finally, we will prove that the map $\chi$ has discrete fibers (see Prop. \ref{discretness})
\subsection{Extending the map $\chi$ to all of $\Lambda$}\label{external}
Let $T$ be a holomorphic tubing for the nice analytic family of parabolic-like maps $\textbf{f}$. Call $c\ll$ the
critical point of $f\ll$ and let $n$ be such that $f\ll^n(c\ll) \in A\ll$, $f\ll^{n-1}(c\ll) \notin A\ll$.
 Lift the holomorphic tubing $T\ll$ to $T_{\lambda,
n-1}$ (see Section \ref{Li}).
 We
can therefore extend the map $\chi$ by setting:
$$\chi:\Lambda \setminus M_f \rightarrow \C \setminus M_1 $$ 
$$ \lambda \rightarrow  \Phi^{-1}\circ T_{\lambda, n-1}^{-1}(c\ll)$$
where
$\Phi: \C\setminus M_1
\rightarrow \C
\setminus \overline \D$ is the canonical isomorphism between the complement of $M_1$
and the complement of the unit
disk (see \cite{M2}).
Since the Tubing $T\ll$ has locally bounded dilatation, the map $\chi:\Lambda \setminus M_f \rightarrow \C \setminus M_1$ is quasiregular on $\Lambda' \setminus M_f $
for any open $\Lambda' \subset \subset \Lambda$. 

\subsection{Indifferent periodic points}
An indifferent periodic point $z'$ for $f_{\lambda_0},\,\,\lambda_0 \in \Lambda$,
is called \textit{persistent} if for each neighborhood $V(z')$
of $z'$ there exists a neighborhood $W(\lambda_0)$ of $\lambda_0$ such that, for every $\lambda \in W(\lambda_0)$ the
map $f_{\lambda}$ has in
$V(z')$ an indifferent periodic point $z'\ll$ of the same period and multiplier (see \cite{MSS}).
Let $(f\ll)_{\lambda \in \Lambda}$ be an family of parabolic-like mappings.
For all $\lambda \in \Lambda$, the parabolic fixed point is persistent. Since all the other
indifferent periodic points are non persistent, in the remainder we will call them
\textit{indifferent periodic points} without further notation.
\begin{prop}\label{neutparametervalues}
The indifferent parameter values for a family of parabolic-like mappings
belong to $\partial M_f$. 
\end{prop}
\begin{proof}
The proof follows the proof of Prop. 11 in \cite{DH}.
Since for all $\lambda \in \Lambda \setminus M_f$ the critical point $c\ll$ of $f\ll$ belongs
to $(U'_{\lambda} \setminus K_{\lambda})$, the map $f\ll$ is hyperbolic.

Assume that for $\lambda_0 \in \mathring{M_f}$ the map $f_{\lambda_0}$ has an
indifferent periodic point $\alpha_0$ of period $k$, and assume first $(f^k)'(\alpha_0)\neq 1$.
By the Implicit Function Theorem there exist $W,\,\,V$ neighborhoods of $\lambda_0$ and $\alpha_0$ respectively,
with $W \subset \mathring{M_f}$, where the indifferent cycle and the critical point move holomorphically with the parameter,
and where the multiplier map $\rho(\lambda)= (f\ll^k)'(\alpha\ll)$ is a holomorphic non constant map. Set $\alpha(\lambda)=\alpha_{\lambda}$.
By taking a restrictions if necessary, we can assume $\lambda_0$ is the only parameter in $W$
for which $f\ll$ has in $V$ an indifferent periodic point.
Let $(\lambda_n)$ be a sequence in $W$ converging to $\lambda_0$, such that for all $n$, $|\rho(\lambda_n)|<1$.
Hence for all $n$, there exists a $\alpha^i(\lambda_n) \in \{\alpha^0(\lambda_n),\,...,\alpha^{k-1}(\lambda_n)\}$ such that
$$f^{i+kp}_{\lambda_n}(c_{\lambda_n}) \rightarrow \alpha^i(\lambda_n) \,\,\, \mbox{as } p \rightarrow \infty$$
We can assume $i$ independent of $\lambda$ by choosing a subsequence.
Let us define for all $\lambda \in W$ the sequence
$$F_p(\lambda)=f^{i+kp}_{\lambda}(c_{\lambda}).$$
Since $(F_p)$ is a family of analytic maps bounded on any compact
subset of $W$, it is a normal family. Let $F_{p_n}$ be a subsequence converging to some function
$h$. 
Then $h(\lambda_n)=\alpha(\lambda_n)$ for all $n$, hence $h=\alpha$ and
for all $\lambda \in W$, $F_p(\lambda) \rightarrow \alpha(\lambda)$.
But in $W$ there are parameters $\lambda^*$ for which
$\alpha(\lambda^*)$ is a repelling periodic point, and thus it cannot
attract the sequence $F_p(\lambda^*)$.

In the case $(f^k)'(\alpha_0)= 1$, let $\Lambda_0$ be a neighborhood of $\lambda_0 \in  \mathring{M_f}$, let
$\lambda: W(0) \rightarrow \Lambda_0$, $t \rightarrow t^2+\lambda_0,$ be a branched covering of $\Lambda_0$ branched at $0$
for some neighborhood $W(0)$ of $0$, and repeat the previous argument.
\end{proof}

\subsection{Continuity of $\chi$ on $\partial M_f$}\label{con}
In this Section we prove that the map $\chi:\Lambda \rightarrow \C $ is continuous
on the boundary of $M_f$.
\begin{prop}\label{7}
Suppose $A_1,\,\,A_2 \in \C$, with $B_1=1-(A_1)^2 \in \partial M_1$. If the maps $P_{A_1}$ and
$P_{A_2}$ are quasiconformally conjugate, then $B_1=B_2$.
\end{prop}
\begin{proof}
Let $(P_1, U', U, \gamma_1)$ and $(P_2, V', V, \gamma_2)$ be parabolic-like restrictions of $P_{A_1}$ and $P_{A_2}$
respectively (for the construction of a parabolic-like restriction of members of the family $Per_1(1)$ see the proof of Prop. 4.2 in \cite{LL}),
and let $\varphi:U \rightarrow V$ be a quasiconformal conjugacy between them. 
If $K_{P_1}$ is of measure zero (where $K_{P_1}= \widehat \C \setminus B_{\infty}$, and $B_{\infty}$ is the parabolic basin of attraction of infinity,
see Section 1 in \cite{LL}), then $\phi$ is a hybrid conjugacy and the result follows from
Prop. 6.5 in \cite{LL}.

Let $K_{P_1}$ be not of measure zero.
Define on $\widehat{\C}$ the following Beltrami form:
$$\widetilde{\mu}(z):= \left\{
\begin{array}{cl}
(\phi)^*\mu_0 &\mbox{on  }  K_{P_1} \\
0 &\mbox{on  } \widehat{\C}\setminus K_{P_1}\\
\end{array}\right.
$$
Since $\phi$ is quasiconformal, $||\widetilde{\mu}||_{\infty}=k<1$. Therefore for $|t|<1/k$ we can define on $\widehat{\C}$ the family of
Beltrami form $\mu_t=\widetilde{\mu} t$, and $||\mu_t||_{\infty}<1$. The family $\mu_t$ depends holomorphically on $t$.
 Let $\Phi_t: \widehat{\C} \rightarrow \widehat{\C}$ be the family of integrating maps fixing $\infty$, $1$ and $0$.
Hence the family $\Phi_t$ depends
holomorphically on $t$, $\Phi_1=\phi$ and $\Phi_0=Id$.
The family of holomorphic maps $F_t=\Phi_t \circ P_{A_1} \circ \Phi_t^{-1}$ has the form
$F_t(z)=z+1/z+ A(t)$ (since it is a family of quadratic rational maps with a parabolic fixed point at $z=\infty$ with
preimage at $z=0$ and a critical point at $z=-1$)
and it depends holomorphically on $t$. Therefore the map $\alpha: t \rightarrow B(t)=1-A^2(t)$ is
holomorphic, hence it is either an open or constant, and $\alpha(0)=B_1 \in \partial M_1$.
If $\alpha(t)$ is open, there exists a neighborhood $W$ of $0$ such that $\alpha(W)
\subset M_1$. Hence the map $\alpha(t)$ is constant, so for all $t$, $\alpha(t)=B_1$. In particular $\alpha(1)=B_1$, and $F_1=P_{A_1}$.

Finally the map $\phi \circ \Phi_1^{-1}$ is a quasiconformal conjugacy between $P_{A_1}$ and $P_{A_2}$ with $(\phi
\circ \Phi_1^{-1})^*\mu_0=\mu_0$ on $K_{P_1}$. So $P_{A_1}$ and $P_{A_2}$ are hybrid equivalent, and the result 
follows by Prop. 6.5 in \cite{LL}.
\end{proof}

\begin{prop}\label{conti}
The map $\chi:\Lambda \rightarrow \C $ is continuous at any point $\lambda \in \partial M_f$,
and moreover $\chi(\lambda) \in \partial M_1$.
\end{prop}
\begin{proof}
In order to prove continuity at any point $\lambda \in \partial M_f$, we have to show that for every sequence $\lambda_n \in \Lambda$
converging to $\lambda_0 \in \partial M_f$ there exists a subsequence $\lambda_{n^*}$ such that $B_{n^*}= \chi(\lambda_{n^*})$
converges to $B_0=\chi(\lambda_0) \in \partial M_1$.
Let us start by proving that $B_0 \in \partial M_1$.

Let $\lambda_m$ be a sequence of indifferent parameters
converging to $\lambda_0$. Hence there exists a sequence $B_m= \chi(\lambda_m)$ such that, for each $m$, $f_{\lambda_m}$ is hybrid
conjugate to $P_{A_m}$ by some quasiconformal map $\phi_m$. The sequence $\phi_m$ is a sequence of quasiconformal maps with locally
bounded dilatation (see Remark \ref{locboundil}), hence it is precompact
in the topology of uniform convergence on compact subsets of $U\l0$ (see \cite{Ly}).
Therefore there exists a subsequence $\phi_{\lambda_{m^*}}$ which converges to 
some quasiconformal limit map $\widetilde{\phi}$,
which conjugates $f\l0$ to some $P_{\widetilde A}$, so $B_{m^*}\rightarrow \widetilde B$.
For all $m$ the map $f_{\lambda_m}$ has an indifferent periodic point,
hence $P_{A_m}$ has an indifferent periodic point,
thus $B_m \in \partial M_1$ and finally $\widetilde B\in \partial M_1$.
Since the map $f\l0$ is hybrid conjugate to $P_{A_0}$ and quasiconformally
conjugate to $P_{\widetilde A}$, and $\widetilde B \in \partial M_1$, by Prop.\ref{7} $B_0= \widetilde B\in \partial M_1$.

Let now $\lambda_n \in \Lambda$ be a sequence converging to $\lambda_0 \in \partial M_f$.
Since the sequence $\phi_n$ is precompact, there exists a
subsequence $(\lambda_n^*)$ such 
that $\phi_{\lambda_n^*}$ converges to some limit map $\widehat{\phi}$ which 
conjugates $f\l0$ to $P_{\widehat A}=\widehat{\phi} \circ f\l0 \circ \widehat{\phi}^{-1},$ so  $B_{n^*}\rightarrow \widehat B$.
Finally, since $B_0 \in \partial M_1$ and $f\l0$ is quasiconformally conjugate to both $P_{\widehat A}$ and $P_{A_0}$, by Prop.\ref{7} $\widehat B=B_0$.

\end{proof}

\subsection{Analicity of $\chi$ on the interior of $M_f$}\label{anal}
The proof of the analycity of the map $\chi$ on the interior of $\mathring{M_f}$ (see \ref{an})
follows the proof Lyubich gave in the setting
of polynomial-like mappings (see \cite{Ly}). We will prove that the map $\chi$ is holomorphic on hyperbolic components
first, and then on queer components.
To prove that $\chi$ is holomorphic on queer components, we first need the following Proposition.
\begin{prop}\label{invariantline}
Let $Q$ be a queer component of $\mathring{ M_1}$. Then for every
$A \in Q$, $J_{P_A}$ admits an invariant Beltrami form with positive support. In particular,
$\mbox{area }J(P_A ) > 0$.
\end{prop}
\begin{proof}
Choose $B_0 \in Q$ and set $P_0=P_{A_0}$.
Let us start by proving that there exists a dynamical holomorphic motion 
$\tilde\eta_A:Q \times\widehat{\C} \rightarrow \widehat{\C}$ with base point $A_0$, holomorphic in $z$.

Let $\Xi^0$ be an attracting petal of $P_0$ containing the critical value $z=2$, and let $\Phi_0: \Xi^0 \rightarrow
\mathbb{H}_r$ be the incoming  Fatou coordinates of $P_0$ normalized by $\Phi_0(2)=1$. For $A \in Q$, let
$\Xi^A$ be an attracting petal of $P_A$ and let $\Phi_A: \Xi^A \rightarrow \mathbb{H}_r$ be the incoming  Fatou
coordinates of $P_A$ with $\Phi_A(2+A)=1$.
The map $\eta_A = \Phi_A^{-1} \circ\Phi_0: \Xi^0 \rightarrow \Xi^A$ is a conformal conjugacy between $P_0$ and $P_A$.
Defining $\Xi^0_{-n},\, n>0$ as the connected component of 
$P_0^{-n}(\Xi^0)$ containing $\Xi^0$, and $\Xi^A_{-n},\, n>0$ as the connected component of
$P_A^{-n}(\Xi^A)$ containing $\Xi^A$, we can lift the map $\eta$ to $\eta_n: \Xi^0_{-n} \rightarrow
\Xi^A_{-n}$.
Since $K_{P_0}$ and $K_{P_A}$ are connected (where $K_{P_0}$ and $K_{P_A}$ are the
complements of the basin of attraction of the parabolic fixed point for $P_0$ and $P_A$ respectively),
by interated lifting we can extend  $\eta_n$ to 
$\eta_A:\widehat{\C} \setminus K_{P_0} \rightarrow \widehat{\C} \setminus K_{P_A}$.
The map $\eta_A$ is a holomorphic conjugacy
between $P_0$ and $P_A$, and since the family $Per_1(1)$
is a family of holomorphic maps depending holomorphic on the parameter and so Fatou coordenates depend holomorphically on the parameter, the family $(\eta_A)_{A \in Q}$ depends
holomorphically on the parameter. Hence $\eta_A: Q \times \widehat{\C} \setminus K_{P_0} \rightarrow \widehat{\C} \setminus K_{P_A}$
is a dynamical holomorphic motion with base point $A_0$, holomorphic in $z$.
Since for every $A \in Q$ all the periodic points of $P_A$ (but the parabolic fixed point)
are repelling, $K_A$ is nowhere dense. Hence $K_{P_0}= J_{P_0}$ and thus by the $\lambda$-Lemma we can extend $\eta_A$ to 
$\tilde \eta_A:Q \times \widehat{\C} \rightarrow \widehat{\C}$. Note that $\tilde \eta_A$ still conjugates dynamics, and it is conformal on
$\widehat{\C} \setminus K_{P_0}$.

Define $\mu_A=(\tilde \eta_A)^*\mu_0$. By construction, $\mu_A=\mu_0$ on $\widehat{\C} \setminus K_{P_0}$. On the other hand,
if $\mu_A = \mu_0$ on $\widehat{\C}$, by the Weyl's Lemma for every $A \in Q$ the map $\tilde \eta_A$ is holomorphic, hence
for all $A \in Q$ the maps $P_A$ are conformally equivalent. Therefore
$\mu_A \neq \mu_0$ on $K_{P_0}$, and thus $\mbox{area}(\mbox{supp} \mu_A)>0$ on $K_{P_0}$. In particular, this implies
$\mbox{area}(J_{P_0})>0$.
\end{proof}

\begin{prop}\label{an}
The map $\chi: \Lambda \rightarrow \C$ depends analytically on $\lambda$ for $\lambda
 \in \mathring{M_f}$.
\end{prop}

\begin{proof}
Let us start by proving that, for every hyperbolic component $P \subset \mathring{M_f}$,
there exists a hyperbolic component $Q \subset \mathring{M_1}$
such that $\chi_{|P}: P \rightarrow Q$ is a holomorphic map.
By the Implicit Function Theorem and Prop. \ref{neutparametervalues},
all the parameter values in $P$ are hyperbolic.
Hence for $\lambda \in P$, $f\ll$ has an attracting cycle, thus $P_{A\ll}= \phi\ll \circ f\ll \circ (\phi\ll)^{-1}$
has an attracting cycle and $Q \subset \mathring{M_1}$. Since $\phi\ll$ is conformal on $\mathring{K\ll}$, calling
$\rho_P,\,\,\rho_Q$ the multiplier maps on $P,\,\,Q$ respectively, $\rho_P(\lambda)= \rho_Q(A\ll)$.
Hence on $P$ we can write the map $\chi$ as $\chi= \rho_Q^{-1} \circ \rho_P$.
Since $f\ll$ is holomorphic in both $\lambda$ and $z$, and by the Implicit Function Theorem the
attracting cycle moves holomorphically on $P$,
the map $\rho_P$ is holomorphic. For the same reason, $\rho_Q$ is holomorphic as well.
Since $\rho_Q$ has degree $1$ (see \cite{PT}), it is conformal and then $\chi$ is holomorphic.

Let now $C$ be a queer component in $\mathring{M_f}$, and let $\lambda_0 \in C$. Since $C \subset M_f$
we can lift the holomorphic motion $\widehat\tau\ll: A\l0 \rightarrow A\ll$ constructed in \ref{holmot}
to $\tau\ll : U\l0 \setminus K\l0 \rightarrow U\ll \setminus K\ll$ (as we did for the holomorphic Tubing, see \ref{Li}).
Since for all $\lambda \in C$,
$K\ll$ is nowhere dense, by the $\lambda$-Lemma we can extend $\tau\ll$ to
a dynamical holomorphic motion $\tau\ll: U\l0 \rightarrow U\ll$.
Let $P_{A\l0}$ be the member of the family $Per_1(1)$ hybrid conjugate to $f\l0$, let $\phi\l0$ be a hybrid conjugacy
between them and set $K_0=K_{P_{A\l0}}$. 
Note that, for all $\lambda \in C$, the map $\tau\ll \circ \phi\l0^{-1}:\phi\l0(U\l0) \rightarrow U\ll$
is a quasiconformal conjugacy between $P_{A\l0}$ and $f\ll$.
Define on $\widehat \C$ the following family of Beltrami forms:
$$\nu\ll(z):= \left\{
\begin{array}{cl}
(\tau\ll \circ \phi\l0^{-1})^* \mu_0 &\mbox{on } K_0 \\
\mu_0 &\mbox{on } \widehat \C \setminus K_0\\
\end{array}\right.
$$
The family $\nu\ll$ is a family of $P_{A\l0}$-invariant Beltrami forms depending holomorphically on $\lambda$.
By Prop. \ref{invariantline}, for every $\lambda \in C$, on $K_0$ $\mbox{area}(\mbox{supp} \nu\ll)>0$.
Let $\psi\ll: \widehat \C \rightarrow \widehat \C$ be the family of integrating maps fixing
$-1,\,0$ and $\infty$, then the family $P_{A(\lambda)}= \psi\ll \circ P_{A\l0} \circ (\psi\ll)^{-1}$
has the form $P_{A(\lambda)} (z)= z + 1/z + A(\lambda)$, where $ A(\lambda)$ depends holomorphically on the parameter.
 Finally, for every $\lambda \in C$, the map $\psi\ll \circ \phi\l0 \circ \tau\ll^{-1}$ is a hybrid
conjugacy between $f\ll$ and $P_{A(\lambda)}$, hence $P_{A(\lambda)}$ is the straightening of $f\ll$ and
the map $\chi_{|C}$ is holomorphic.
\end{proof}

\subsection{Discreteness of fibers}\label{d}
Set $\mathcal{B}= \chi(\Lambda)$.
\begin{prop}\label{discretness}
 For every $B \in \mathcal{B}$, $\chi^{-1}(B)$ is discrete.
\end{prop}

\begin{proof}
This follows the proof of Lemma $10.13$ in \cite{Ly}.
Let us assume there exists a $B \in  \mathcal{B}$ such that
there exists a sequence $\lambda_n \in \chi^{-1}(B)$ and $\lambda_n \rightarrow \widehat \lambda \in \chi^{-1}(B)$.
The map $\chi$ is quasiregular on $\Lambda \setminus M_f$ and holomorphic
on $\mathring{M_f}$, hence $\widehat \lambda \in \partial M_f$
(or $\widehat \lambda$ in a queer component of $\mathring {M_f}$ for which $\chi$ is constant, and then we can replace $\widehat \lambda$
with a boundary point).

Note that, for every $n$, the maps $f_{\lambda_n}$ are hybrid equivalent to $f\wwl$ by some hybrid equivalence
$\beta_{\lambda_n}$. Let us assume that for all $\lambda$ in a neighborhood of $\wl$,
$f_{\lambda}^{-1}(\widetilde \Delta_{\lambda}) \subset \Delta_{\lambda}$
(in other case, take a nice analytic family of parabolic-like restrictions for which the assumption
holds).
For every $\lambda$, call $z\ll$ the parabolic fixed point of $f_{\lambda}$, $c\ll$ its critical
point and $v\ll$ its critical value.
Consider $\wl$ as the base point of a holomorphic motion $\widehat \tau: \Lambda \times  A\wwl \rightarrow A\ll$,
extend it by the Slodkowski's
Theorem and then restrict it to a holomorphic motion
$\tau: \Lambda \times \widehat \C \setminus \widetilde\Omega\wwl' \rightarrow \widehat \C \setminus \widetilde\Omega\ll'$.
Define for every $n$ the map $\alpha_{\lambda_n} : \widehat \C \rightarrow \widehat \C$ as follows:

$$\alpha_{\lambda_n}(z):= \left\{
\begin{array}{cl}
\tau_{\lambda_n} &\mbox{on } \widehat\C \setminus \widetilde\Omega\wwl'\\
 \widetilde f_{\lambda_n}^{-n} \circ \tau_{\lambda_n} \circ \widetilde f\wwl^n &\mbox{on } A_{\wl,n} \\
\beta_{\lambda_n} &\mbox{on  } K\wwl\\
\end{array}\right.
$$
where the maps $\widetilde  f_{\lambda}$ are as in \ref{holmot} and the sets $A_{\wl,n}$ are
constructed in \ref{Li}.
The proof of Prop. 6.4 in \cite{LL} shows that, for every $n$, the map $\alpha_{\lambda_n}$ is
continuous and hence quasiconformal. Therefore, for every $n$,
$\alpha_{\lambda_n}$ restricts to a hybrid equivalence between $f\wwl$ and $f_{\lambda_n}$.
Consider on $\widehat \C$ the family of Beltrami forms $\nu_{\lambda_n}=(\alpha_{\lambda_n})^*\mu_0$.
Note that trivially $\alpha_{\lambda_n}$ integrates $\nu_{\lambda_n}$, and for
some subsets $O_n$ of $\widetilde\Omega'_{\lambda_n}$ and $O$ of $\widetilde\Omega'_{\wl}$,
$(f_{\lambda_n})_{|O_n}= \alpha_{\lambda_n} \circ f\wwl \circ (\alpha_{\lambda_n})^{-1}_{|O}$.

On the other hand,
define on $\widehat \C$ the family of Beltrami forms $\mu\ll$ as follows:
$$\mu\ll(z):= \left\{
\begin{array}{cl}
\mu_{\lambda,0}=(\tau\ll)^* \mu_0 &\mbox{on } \widehat \C \setminus \widetilde\Omega\wwl' \\
(\widehat{f}_{\widehat\lambda}^n)^*\mu_{\lambda,0} &\mbox{on } \widehat A_{\wl,n}\\
0 &\mbox{on } K_{\wl}\\
\end{array}\right.
$$
where for every $\lambda$ the map $ \widehat{f}_{\widetilde\lambda}$ which defines the sets $\widehat A_{\wl,n}$ and
spreads $\mu\ll$ is defined as follows:
 $$ \widehat{f}_{\widetilde\lambda}(z)=\left\{
\begin{array}{cl}
\tau\ll^{-1} \circ f_{\lambda_n} \circ \tau\ll &\mbox{on } \tau\ll^{-1}(f\ll^{-1}(\widetilde{\Delta_{\lambda}}))\\
f_{\widehat\lambda}  &\mbox{on  }  \widetilde\Omega_{\widehat\lambda}'\\
\end{array}\right.
$$ 
Note that $ \widetilde{f}_{\widetilde\lambda}$ and $ \widehat{f}_{\widetilde\lambda}$ coincide
on $\widetilde\Omega\wwl'$, hence for every
$n,( \bigcup_n A_{\wl,n})\cap \widetilde\Omega\wwl'=( \bigcup_n\widehat A_{\wl,n})\cap \widetilde\Omega\wwl'$.
Therefore, for all $n$, $\mu_{\lambda_n}=\nu_{\lambda_n}$.

The family $\mu_{\lambda}$ depends holomorphically on $\lambda$,
because $\tau\ll$ is a holomorphic motion, on $\Delta\wwl$ it is constant and
on $\widetilde\Omega'\wwl$ it is spread by the dynamics
of $f\wwl$ (which does not
depends on $\lambda$).

Let $H\ll: \widehat \C \rightarrow \widehat \C$ be the holomorphic family of integrating maps
mapping $z\wwl$ to $z\ll$, $c\wwl$ to $c\ll$ and $v\wwl$ to $v\ll$.
The family $G_{\lambda}=H\ll \circ f\wwl \circ  H\ll^{-1}: H\ll(U'\wwl) \rightarrow H\ll(U\wwl)$
is a holomorphic family of parabolic-like mappings hybrid equivalent to $f\wwl$.
For all $n$, $\alpha_{\lambda_n}=H_{\lambda_n}$ (since they are solutions of the same family of Beltrami form and
they coincide on $3$ points), and therefore on the $O_n$,
$G_{\lambda_n}(z)=f_{\lambda_n}(z)$.

Choose a neighborhood $\Lambda_*$ of $\wl$ in $\Lambda$ and an open set $V$ in $\cap_n O_n$
such that the maps $G\ll$ and $f\ll$ are well-defined in $\Lambda_* \times V$.
Since $G_{\lambda}(z)=f_{\lambda}(z)$ on $\lambda_n \times V$, and $\lambda_n$ converges to $\wl$,
$G_{\lambda}=f_{\lambda}$ on $\Lambda_*\times V$.
This is impossible, because in $\Lambda_*$ there are $\lambda$ for which $f\ll$ has disconnected Julia set
(since $\wl \in \partial M_f$), while for all $\lambda$, $G\ll$ has connected Julia set
(since it is equivalent to $f\wwl$).

\end{proof}
\section{The map $\chi: \Lambda \rightarrow \C$ is a ramified covering from the connectedness locus $M_f$ to
$M_1\setminus\{1\}$}\label{final}
In this chapter, we will first prove that for any closed and connected
subset $K$ of $\mathcal{B}= \chi(\Lambda)$, if $C= \chi^{-1}(K)$ is compact,
then $\chi_{|C}$ is a proper map of degree $d$ (Prop. \ref{th}),
and then that $\chi: C \rightarrow K$ is a $d$-fold branched covering (Prop. \ref{cov}).
Finally, we will prove that, under certain conditions (Def. \ref{properfam}), for every
neighborhood $U$ of the root of $M_1$ (without specifications, we consider a neighborhood open),
the set $\chi^{-1}(M_1 \setminus U)$ is compact in $\Lambda$.
This implies Theorem \ref{bigthm}.

Denote by $i_{\lambda}(\chi)$ the local degree of $\chi$ at $\lambda$. Note that,
since $\chi: \Lambda \rightarrow \mathcal{B}$ is quasiregular on $\Lambda \setminus M_f$ and holomorphic on $\mathring{M_f}$,
for all $\lambda \in \Lambda$, $i_{\lambda}(\chi)>0$.

\begin{prop}\label{th}
 Let $K$ be a closed and connected subset of $\mathcal{B}$, and $C= \chi^{-1}(K)$. If $C$ is compact, then
 there exist neighborhoods $\hat V$ of $K$ in $\mathcal{B}$ and $\hat U$ of $C$ in $\Lambda$ such that
 $\chi: \hat U \rightarrow \hat V$ is a proper map of degree $d$, 
 where, for every $y \in K$, $d = \sum_{x \in \chi^{-1}(y)} i_x(\chi)$. 
\end{prop}
\begin{proof}
The proof follows the analogous one in \cite{DH}.
Let $N$ be a closed neighborhood of $C$ in $\Lambda$ with $dist(C,\partial N)>0$ (it exists since $C$ compact).
Hence $C \subset  N  \subset \Lambda$, $\chi:N \rightarrow  \chi(N)$ is proper and
$\partial K \cap \chi(\partial N) =\emptyset$.
Call $\hat V$ the connected component of $\mathcal{B} \setminus  \chi(\partial N)$ which contains $K$, and
set $\hat U= \chi^{-1}(\hat V) \cap N$.    
Then $\chi^{-1}(\hat V) \cap \partial N= \emptyset$, hence 
the map $\chi_{|\hat U}:\hat U \rightarrow \hat V$ is proper.

Since $\chi$ is continuous and $\hat V$ is
connected, $\hat U$ is the union of connected components. Set $\hat U= \bigcup_j \hat U_j$. The
restriction $\chi: \hat U_j \rightarrow \hat V$ is then a proper map between connected sets, thus it has a degree
$d_j$, and since $\chi$ has discrete fibers, for every $v \in \hat V$
(see \cite{H} pg. 136): $$d_j=\sum_{u \in \chi^{-1}(v)\cap \hat U_j} i_u(\chi) $$ (note that $d_j >0$). 
Therefore $\chi: \hat U \rightarrow \hat V$ has a degree $d= \deg \chi_{|\hat U}= \sum_j d_j$ 
and in particular for all $y \in K,\,\,\,d= \deg \chi_{|\hat U}= \deg \chi_{|C}= \sum_{x \in \chi^{-1}(y) \cap C} i_x(\chi).$

\end{proof}
\begin{prop}\label{cov}
In the hypothesis of Prop. \ref{th}, the map $\chi_{|\hat U}:\hat U \rightarrow \hat V$ is a branched
covering of degree $d$.

\end{prop}

\begin{proof}
The map $\chi_{|\hat U}:\hat U \rightarrow \hat V$ is continuous, and
by the previous proposition it is a proper surjective map of degree $d$.
Let $y \in \hat V$, and let $Y$ be a neighborhood of $y$ in $\hat V$ such that for all $x \in X= \chi^{-1}(Y)$ and
$x \neq \chi^{-1}(y)$, $x$ is a regular point (such a $Y$ exists since the fiber of $\chi$ are finite).
Hence $X= \bigcup_{j \in J} U_j$, with the $U_j$ disjoint
and $1\leq j \leq d$. If $j=d$, $y$ is a regular point, and for all $j$, $\chi_{|U_j}$ is a homeomorphism.

If $j < d$, we want to show that for every $w \in Y \setminus \{y\}$ there exists a neighborhood $W \subset Y$ of $w$
such that $\chi^{-1}(W)= \bigcup_{1 \leq j \leq d} T_j$, with the $T_j$ disjoint and $\chi_{|T_j}$ homeomorphism.
This is clear because all the points in $X$ different from the preimage of $y$ are regular points.

\end{proof}

\subsection{Proper families of parabolic-like maps}\label{p}

As we saw in \ref{anfam}, the range $\mathcal{B}$ of the map $\chi$ is not the whole of $\C$ but
a proper subset, because there is no $\lambda \in \Lambda$ for which $f\ll$ is hybrid equivalent to $P_0=z+1/z$.
Hence $M_1 \nsubseteq \mathcal{B}$. However, we could hope that $B=1$ is the only
point of $M_1$ missing from $\mathcal{B}$, or in other words, that as $B \rightarrow \partial \mathcal{B}$
$B \notin M_1$ or
$B \rightarrow 1$.
Indeed this is the case under appropriate conditions (e.g. the following one).
\begin{defi}\label{properfam}
Let $\textbf{f}=(f_{\lambda}: U_{\lambda}' \rightarrow U_{\lambda} )_{\lambda \in \Lambda} $ be a nice analytic family of 
parabolic-like maps of degree $2$, such that, for $\lambda \rightarrow \partial \Lambda$:
\begin{enumerate}
 \item $\lambda \notin M_f$ or
 \item $\chi(\lambda) \rightarrow 1$.
\end{enumerate}
Then we call $\textbf{f}$ a \textit{proper family of parabolic-like mappings}.
\end{defi}
\begin{prop}
 Let $\textbf{f}$ be a proper family of parabolic-like mappings. Then, for every $U(1)$ neighborhood of  $1$ in $\C$,
setting $K= M_1 \setminus U(1)$, the set $C=\chi^{-1}(K)$ is compact in $\Lambda$.
\end{prop}
\begin{proof}
 Assume $C$ is not compact in $\Lambda$. Then there exists a sequence $(\lambda_n) \in C$ such that
$\lambda_n \rightarrow \partial \Lambda$ as $n \rightarrow \infty$.
On the other hand, for all $n$, $\chi(\lambda_n) \in K$. Let $\chi(\lambda_{n_k})$ be a subsequence converging to some
parameter $B$. Since $K$ is compact, the limit point $B$ belongs to $K \subset M_1 \setminus \{1\}$. This is a
contradiction, because $\textbf{f}$ is a proper family of parabolic-like mappings. Therefore $C$ is compact in
$\Lambda$.
\end{proof}

Hence if $\textbf{f}$ is a proper family of parabolic-like mappings, $U(1)$ a neighborhood of $B=1$,
$K= M_1 \setminus U(1)$, and $C=\chi^{-1}(K)$, by Prop. \ref{cov} there exist
neighborhoods $\hat U$ of $C$ in $\Lambda$ and $\hat V$ of $K$ in $\mathcal{B}$ such that the restriction
$\chi: \hat U \rightarrow \hat V$ is a $\mathcal{D}$-fold branched covering.
 Next proposition tells us how to compute the
degree $\mathcal{D}$ of the covering.

\begin{prop}
 Let $\textbf{f}$ be a proper family of parabolic-like mappings, $U(1)$ a neighborhood of $B=1$
 (with $0 \notin U(1)$),
$K= M_1 \setminus U(1)$, and $\hat V$ a neighborhood of $K$ given by Prop. \ref{th}.
Then calling $c\ll$ the critical point of $f\ll$ and $\hat U= \chi^{-1} (\hat V) $, the degree $\mathcal{D}$ of
the branched covering $\chi: \hat U
\rightarrow \hat V$ is
equal
to the
number of times $f\ll(c\ll)-c\ll$ turns around $0$ as $\lambda$ describes $\partial C$. 
\end{prop}
Let us remember that for every $A$ the map $P_A=z + 1/z + A$ has two critical points: $z=1$ and $z=-1$. After a change
of coordenates we can assume $z=1$ is the first critical point attracted by $\infty$. Hence for all $A\in \C$, $z=-1$
is the critical point of the parabolic-like restriction of $P_A$ (see the proof of Prop. 4.2 in \cite{LL}).
\begin{proof}
The proof follows the analogous one in \cite{DH}.
Let $c\ll$ be the critical point of $f\ll$. Choose $\lambda_0$ such that $f\l0(c\l0)=c\l0$. Let 
$P_{A_0}$ be the member of the family $Per_1(1)$ hybrid equivalent to $f\l0$. Therefore 
$P_{\pm A_0}(-1)=-1$, hence $\chi(\lambda_0)=0$. This means that
the
multiplicity of $\lambda_0$ as zero of the map $\lambda \rightarrow f\ll(c\ll)-c\ll$ is the multiplicity
of $\lambda_0$ as zero of the map
$\lambda  \rightarrow
\chi(\lambda)$. This last one is $ \sum_{\lambda \in \chi^{-1}(0)} i\ll(\chi)= \mathcal{D}$.
\end{proof}
\subsubsection{Extension to the root}
 Let $(f\ll)_{\lambda \in \Lambda}$ be a proper family of parabolic-like mappings. For every $\lambda \in \Lambda$,
 $f\ll$ is the restriction of some map $F\ll$. Consequently, $\Lambda$ is the restriction of the parameter plane
 of the maps $F\ll$, call it $G$. Call $M_F$ the connectedness locus of $F\ll$, hence $M_f \subset M_F$
\begin{teor}
 Let $\textbf{f}$ be a proper family of parabolic-like mappings. If the map $\chi:M_f \rightarrow M_1 \setminus \{1\}$
 is a homeomorphism, and $\partial \Lambda \cap \partial M_f \subset G$,
 then $\chi$ extends to a homeomorphism $\chi:M_f \cup \{\lambda_*\}\rightarrow M_1$ for a
 unique $\lambda_* \in \partial \Lambda$. More generally, if the the map $\chi:M_f \rightarrow M_1 \setminus \{1\}$
is a degree $\mathcal{D}$ branched covering, and $\partial \Lambda \cap \partial M_f \subset G$, then map $\chi$ extends
continuously to $\chi:M_f \cup \{\lambda_1\}\cup..\cup\{\lambda_{\mathcal{D}}\}\rightarrow M_1 $
for exactly $\mathcal{D} $ points in $ \partial \Lambda$.

\end{teor}
\begin{proof}
Let $\textbf{f}$ be a proper family of parabolic-like mappings for which the map $\chi:M_f \rightarrow M_1 \setminus \{1\}$
is a degree $\mathcal{D}$ covering and $\partial \Lambda \cap \partial M_f \subset G$.
Since $\textbf{f}$ is a proper family, as $\chi(\lambda) \rightarrow 1$, $\lambda \rightarrow \partial \Lambda \cap \partial M_f$.
We will prove that for every $\lambda \in \partial\Lambda \cap \partial M_f$, $\chi(\lambda)=1$, and that $\partial\Lambda \cap \partial M_f$
is a discrete set. Then by continuity, $\partial\Lambda \cap \partial M_f=\{\lambda_1,..,\lambda_{\mathcal{D}}\}$.

The original family $F\ll$
has a persistent parabolic fixed point of multiplier $1$ and it depends
holomorphically on $\lambda$.
Take a succession $\lambda_i \in \mathring{M_f}$ such that $\chi(\lambda_i) \rightarrow 1$, and call $\lambda_*$ the limit of the $\lambda_i$ in $G$.
Since for every $i$, $\lambda_i$ is a hyperbolic parameter, the limit $\lambda_*$ is a hyperbolic or indifferent parameter. 
So, if $\chi(\lambda_*)\neq 1$, $F_{\lambda_*}$ presents a degree $2$ parabolic-like restriction and $\lambda_*\in \Lambda$.
  Since $\lambda_* \notin \Lambda$, $\chi(\lambda_*)=1$. 
 
Let us prove now that the set $\partial\Lambda \cap \partial M_f$ is discrete. Note that this is the set of parameters for which the parabolic fixed point $z\ll$ of $F\ll$
has parabolic multiplicity $n+1$, where $n$ is the multiplicity of $z\ll$ for $\lambda \in M_f$. Then, in a neighborhood of $z\ll$ we can consider $F\ll$ as
 $$z+a_{\lambda} z^{n+1}+ \mbox{h.o.t.},$$
with $n \geq 1$ and $a_{\lambda}$ holomorphic in $\lambda$. Hence the set $\lambda_j^*$ for which $a_{\lambda_j^*}=0$
is a discrete set.

\end{proof}

\subsection{The parameter plane of the family $C_a(z)=z+~az^2+~z^3$ presents baby-$M_1$}
Let us show that the family of parabolic-like mappings $(C_a(z)=z+az^2+z^3)_{a \in \Lambda}$ is proper.
Call $M_a$ the connectedness locus of $(C_a)_{a \in \Lambda}$. The finite
boundaries of $\Lambda$ are the external rays of angle $1/6$ and $2/6$, which
cannot intersect the connectedness locus $M_a$ in other point than the landing point, if
they land. Since these rays land at $a=0$
(see \cite{N}), for $a \rightarrow \partial \Lambda$ either
$a \notin M_a$, hence $B \notin M_1$, or $a \rightarrow 0$, hence $B \rightarrow 1$.

Finally, by the relation $\Phi(a)=\varphi (c_c(a)) $ between external rays in dynamical and parameter plane, the
degree of $\chi_{|M_a}$ is $1$. Therefore
$\mathcal{C}$ presents a baby $M_1$. 
By symmetry, we can repeat the construction for the family $(C_a(z)=z+az^2+z^3)_{a \in \Lambda'}$, where $\Lambda'$
is the open set bounded
by the external rays of angle $4/6$ and $5/6$. Hence the connectedness locus of the family
$(C_a(z)=z+az^2+z^3)_{a \in \C}$ presents two baby $M_1$, namely in the connected component bounded by the external rays of angle $1/6$
and $2/6$, and in the connected component bounded by the external rays of angle $4/6$ and $5/6$ (see Fig. \ref{M}
and \ref{M1} in the Introduction).

\end{document}